\documentclass{mathincs}
\usepackage[T1]{fontenc}
\include{braket}
\usepackage{subfigure}
\usepackage{float}

 \newtheorem{theorem}{Theorem}[section]
 
 \newtheorem{lemma}[theorem]{Lemma}
 \newtheorem{proposition}[theorem]{Proposition}
 \theoremstyle{definition}
 \newtheorem{definition}[theorem]{Definition}
 \theoremstyle{remark}

 \numberwithin{equation}{section}
\usepackage{graphicx}
\usepackage{amsfonts}
\usepackage{amsmath}
\usepackage{amssymb}
\usepackage{url}
 \usepackage[shortlabels]{enumitem}
\usepackage{caption}

\begin{document}

\title[Robust Hadamard matrices]{Robust Hadamard matrices, \\
unistochastic rays in Birkhoff polytope \\
and equi-entangled bases in composite spaces}

\author{Grzegorz Rajchel}
\address{Wydzia{\l} Fizyki, Uniwersytet Warszawski, ul.~Pasteura 5, 02-093~Warszawa, Poland\\
Centrum Fizyki Teoretycznej PAN, Al.~Lotnik{\'o}w 32/44, 02-668~Warszawa, Poland}
\email{grzegorz.rajchel@student.uw.edu.pl}
\author{Adam G\k{a}{}siorowski}
\address{Wydzia{\l} Fizyki, Uniwersytet Warszawski, ul.~Pasteura 5, 02-093~Warszawa, Poland}
\email{adam.gasiorowski@student.uw.edu.pl}
\author{Karol {\.Z}yczkowski}
\address{Instytut Fizyki im. Smoluchowskiego, Uniwersytet Jagiello{\'n}ski\\ ul.~\L{}ojasiewicza 11, 30-348~Krak{\'o}w, Poland\\
Centrum Fizyki Teoretycznej PAN, Al.~Lotnik{\'o}w 32/44, 02-668~Warszawa, Poland}
\email{karol.zyczkowski@uj.edu.pl}

\date{June 16, 2018}

\maketitle

\begin{abstract}
We study a special class of (real or complex) robust Hadamard matrices, 
distinguished by the property that their projection 
onto a $2$-dimensional subspace forms a Hadamard matrix.
It is shown that such a matrix of order $n$ exists, 
if there is a skew Hadamard matrix of a symmetric
conference matrix of this size.
This is the case for any even $n\le 20$,
and for these dimensions we demonstrate that
a bistochastic matrix $B$ 
located at any ray of the Birkhoff polytope,
(which joins the center of this body with any permutation matrix),
is unistochastic.
An explicit form of the corresponding unitary matrix $U$, 
such that $B_{ij}=|U_{ij}|^2$, is determined by a
robust Hadamard matrix. These unitary matrices allow
us to construct a family of orthogonal bases in the composed Hilbert space
of order $n \times n$. Each basis 
consists of vectors with the same degree of entanglement
and the constructed family interpolates between 
the product basis and the maximally entangled basis.
In the case $n=4$ we study geometry of the set 
${\mathcal U}_4$ of unistochastic matrices, 
conjecture that this set is star-shaped
and estimate its relative volume
in the Birkhoff polytope  ${\mathcal B}_4$.

\end{abstract}

 \hskip 4.0cm {\sl Dedicated to the memory of Uffe Haagerup (1949-2015)}

\section{Introduction}

Hadamard matrices with a particular structure attract a lot of 
attention, as their existence is related to several problems in 
combinatorics and mathematical physics. Usually one poses a question,
whether for a given size $n$ there exists a Hadamard matrix with a certain
symmetry or satisfying some additional conditions.
Such a search for structured Hadamard matrices
can be posed as well in the standard case
of real Hadamard matrices \cite{Horadam07},
or in the more general complex case \cite{Cr91,TZ06,Sz11}.

To simplify investigations one studies equivalence 
classes of Hadamard matrices, which differ only by permutations
of rows and columns and by multiplication
by diagonal unitary matrices. 
It was shown by Haagerup \cite{Ha96} that for dimensions $n = 2, 3$ and $n = 5$ all complex Hadamard matrices are equivalent while for $n= 4$ there exists a one parameter family of inequivalent matrices.

In the recent paper by Karlsson \cite{karlson},
aimed to construct new families of complex Hadamard matrices
of size $n=6$, the author studied $H_2$-reducible matrices
defined as Hadamard matrices with the property that
all their blocks of size $2$ form Hadamard matrices.

In this work we introduce a related, but 
different notion of {\sl robust Hadamard} matrices. 
This class contains a Hadamard matrix $H_n$, 
such that each of its principal minors of order two forms 
a Hadamard matrix $H_2$. The name refers to the fact that the  
Hadamard property is robust with respect to projection,
as any projection of $H_n$ onto a two dimensional subspace spanned 
by the elements of the standard basis is a Hadamard matrix.

The notion of robust Hadamard matrices will be 
useful to broaden our understanding of the problem of unistochasticity 
inside the Birkhoff polytope \cite{Bi46}. 
For any bistochastic matrix $B$ one asks 
whether there exist a unitary matrix $U$, 
such that $B_{ij}=|U_{ij}|^2$. In the simplest case
$n=2$ any bistochastic matrix is unistochastic,
for $n=3$ the necessary and sufficient conditions 
for this property are known \cite{AYP79,JS88,Na96}, in order $n=4$
we provide an explicit construction in Appendix A,
while for $n\ge 5$ the unistochasticity
problem remains open \cite{AMM91,BEKTZ05}.

A robust Hadamard matrix $H_n$ will be used to 
to find unitary matrices of order $n$,
designed to prove unistochasticity of bistochastic matrices 
located at any ray of the Birkhoff polytope. 
The key idea behind this construction  -- a decomposition of bistochastic
matrix of an even dimension into square blocks of size two --
was used in the algorithm of Haagerup to analyze the 
unistochasticity problem for $n=4$.
A family of unistochastic matrices
joining a permutation matrix with the flat bistochastic matrix $W_n$
of van der Waerden allows us to construct 
a family of orthogonal bases in the composed Hilbert space ${\mathcal H}_n \otimes {\mathcal H}_n$,
which contains equi-entangled vectors and
interpolates between the product basis 
and the maximally entangled basis \cite{We01}.
Although such "equi-entangled bases" 
have already appeared in the literature \cite{GL10,KM06},
the present construction is simpler
as it corresponds to straight lines
in the  Birkhoff polytope ${\mathcal B}_n$.

This work is organized as follows. In Section \ref{section:robust}
we review the necessary notions, introduce robust Hadamard matrices
 and demonstrate their existence if skew Hadamard matrices or symmetric conference matrices exist.
In Section \ref{section:robust_and_unistochasticity} we present the unistochasticity problem
and show how robust Hadamard matrices can be used to
prove unistochasticity of bistochastic matrices at a ray
joining the center of the Birkhoff polytope with any of its corners. The notion of complementary permutation
matrices is used in Section \ref{section:triangles} to
demonstrate unistochasticity of certain subsets of the 
Birkhoff polytope of an even dimension $n$ for which 
robust Hadamard matrices exist.
Unitary matrices of an even order $n\le 20$ 
corresponding to rays in the Birkhoff polytope
are applied to construct in Section \ref{section:equi} a family of equi-entangled bases
for the composite systems of size $n \times n$.
In Section \ref{section:set_dimension4} we analyze the unistochasticity problem for $n = 4$: using the Haagerup procedure, presented
in Appendix A,
we  investigate the geometry of the set $\mathcal{U}_4$ of unistochastic matrices of order $n = 4$.
An estimation of the relative volume of the set  $\mathcal{U}_4$  with respect to the standard Lebesgue measure is presented in Appendix B.
In Appendix C it is shown that robust Hadamard matrices do not exists for $n=3$ and $n=5$, 
while extensions of these results for any odd $n$, suggested by the referee, are presented in Appendix D.
	
\section{Robust Hadamard Matrices}\label{section:robust}
		We are going to discuss various subclasses of the set of Hadamard matrices. Let us start with basic definitions. Let $M_n(\mathbb{R})$ and $M_n(\mathbb{C})$ denote set of all real (complex) matrices of order $n$.
        
        In analogy to the standard notion of (real) Hadamard matrices
one defines \cite{Cr91,TZ06} 
complex Hadamard matrices.
	\begin{definition}\label{definition:hadamard}
	Matrices of order $n$ with unimodular entries, unitary up to a scaling factor,
		\begin{equation*}
			\mathcal{H}(n)=\{H \in M_n(\mathbb{C}) :\,(HH^{*}=n \mathbb{I},
|H_{ij}|= 1,  i,j=1,\dots,n )  \},
		\end{equation*}
		are called Hadamard matrices. If all entries of the matrix $H$ are real, it is called a real Hadamard.
	\end{definition}

	\begin{definition}\label{definition:robust}
		Set of robust Hadamard matrices for dimension $n$
		\begin{equation*}
			\mathcal{H}^{R}(n) = 
          \mathcal{f}H \in M_n(\mathbb{C}):\,\forall_{i,j\in1,2,...,n,\,i\neq j}
         \begin{pmatrix}H_{ii} & H_{ij}\\
			H_{ji} & H_{jj}
	\end{pmatrix}\in\mathcal{H}_{\mathbb{C}}(2)\}.
		\end{equation*}
	\end{definition}

The name refers to the fact that the Hadamard property 
is robust with respect to projections:
 $H^R$ remains Hadamard after a projection  $\Pi_2$
onto a subspace spanned by two vectors of the basis used,	
$\Pi_{2}H^{R}\Pi_{2} \in \mathcal{H} (2)$.

	Robust Hadamard matrices of order $n$ will be denoted by $H^{R}_{n}$ or $H^{R}$. 
Equivalent definition is:
	\begin{quote}
		Robust Hadamard matrix is a Hadamard matrix,
 in which all principal
minors of order two are extremal so that their modulus is equal to 2.
	\end{quote}
The notion of Hadamard matrices robust with respect to projection can also be used in the complex case. 

\subsection{Equivalence relations}
Let ${\mathcal P}(n)$ or $\mathcal{P}$ denote the set of all permutation matrices $P$ of order $n$. Permutation matrices are used to introduce   
the following equivalence relation in the set of Hadamard matrices.

	\begin{definition}
		Two Hadamard matrices are called \emph{equivalent}, written $H_1\sim H_2$, if there exist unitary diagonal matrices $D_{1},\,D_{2}$ and permutation matrices $P_{1},\,P_{2}\in\mathcal{P}$ such that 
		\begin{equation}
			H_{1} \sim  H_{2}=P_{1}D_{1}H_{1}D_{2}P_{2}.
            \label{eq:equi_relations1}
		\end{equation}
	\end{definition}

In dimensions $1, 2, 4, 8$ and $12$ all real Hadamard matrices are equivalent. For $n=16$ there exist five equivalence classes  \cite{Horadam07} and this number grows fast \cite{Or08} with $n$. 
It is also convenient to distinguish a finer notion
of equivalence with respect to signs.

	\begin{definition}
		Two Hadamard matrices are equivalent with respect to signs, 
written $H_1 \approx H_2$,
 if one is obtained from the other by negating rows or columns,
	\[
	H_{1} \approx  H_{2}=D_{1}H_{1}D_{2},
	\]
    where $D_{1},\,D_{2}$ are orthogonal diagonal matrices.
    \label{eq:equi_relations2}
	\end{definition}
	By definition, if $H_{1}  \approx  H_{2}$ then $H_{1} \sim H_{2}$. 

\subsection{Basic properties of robust Hadamard matrices}

Some basic properties of robust Hadamard matrices are: 
\begin{itemize}
	\item By construction every $H_2$ is robust; \; $\mathcal{H}(2)
 \equiv  \mathcal{H}^{R}(2)$.
	\item For $n \geq 4$ not every $H$ is robust; \;  $\mathcal{H}(n) \neq \mathcal{H}^{R}(n),\,n \geq 4$.
	\item If $H$ is robust then its transpose $H^T$ is also robust.
	\item Real robust matrices form sign-equivalence classes within the real Hadamard equivalence classes. Thus, for the orders with one equivalence class ($n=1,2,4,8,12$) all real robust matrices are equivalent.
\end{itemize}

	\subsection{Skew Hadamard and robust Hadamard matrices}
		\begin{definition}
			A real Hadamard matrix $H$ is a skew Hadamard matrix, written $H\in\mathcal{H}^{S}(n)$, if and only if: 
			\[
				H+H^{T}=2\mathbb{I}.
			\]
		\end{definition}
		\begin{lemma}
			Every skew Hadamard matrix is robust Hadamard.
            \label{lemma:all_robust}
		\end{lemma}
		\begin{proof}
 	By definition, every diagonal element of a skew 
Hadamard matrix  $H$ is equal to $1$. Furthermore, any pair of off--diagonal elements, 
with the number of row and column exchanged, consist of two entries with the opposite sign.
Hence  the determinant of every principal 
minor of order two is equal to $2$, which means that $H$ is a robust Hadamard.
		\end{proof}
        
		\begin{lemma}\label{lemma:equiv}
	Any robust Hadamard matrix $H^{R}$ is sign-equivalent to a certain skew Hadamard
	matrix: $H^{S} \approx H^{R}$. 
		\end{lemma}
\begin{proof}
To show this observe that by multiplying 
the columns of a robust Hadamard matrix $H^{R}$ which contain negative entries at the diagonal by $(-1)$ 
we obtain a skew matrix.
Note that multiplication of columns does not take $H^{R}$
out of the set of robust Hadamard matrices.
\end{proof}

	One can easily see that the following Hadamard matrix of order four is robust and is equivalent to a skew Hadamard:
		\[
			H^{R}_{4}=\begin{pmatrix}1 & 1 & 1 & 1\\
			1 & -1 & -1 & 1\\
			1 & 1 & -1 & -1\\
			1 & -1 & 1 & -1
			\end{pmatrix}
 \approx  
 H^{R,S}_{4} =
    \begin{pmatrix}1 & -1 & -1 & -1\\
			1 & 1 &  1 & -1\\
			1 & -1 & 1 & 1\\
			1 & 1 & -1 & 1
			\end{pmatrix}.
		\]

To show that the set of $H_2$-reducible matrices introduced by Karlsson \cite{karlson}
differs from the set $\mathcal{H}^{R}(n)$ of robust Hadamard matrices discussed here, consider
the following matrix of size $n=4$,
\[
\begin{pmatrix}
             1 & 1 & 1 & 1 \\
             1 & -1 & 1 & -1 \\
             1 & 1 & -1 & -1\\
             1 & -1 & -1 & 1\\
\end{pmatrix}.
\]
It is easy to see that the above matrix is
$H_2$-reducible but is not robust, so these two sets do differ.

\subsection{Symmetric conference matrices and robust Hadamard matrices}

\begin{definition}
	A symmetric matrix with entries $\pm 1$ outside the diagonal and $0$ at the diagonal,
     which satisfy  orthogonality relations:
	\begin{equation*}
		CC^{T} = (n-1) \mathbb{I},
	\end{equation*}
	is called a symmetric {\sl conference} matrix.
\end{definition}
	
	\begin{lemma}
		Matrices of the structure:
		\begin{equation*}
			H^R = C + i \mathbb{I},
		\end{equation*}
		where $C$ is symmetric conference matrix, are complex robust Hadamard matrices.
	\end{lemma}
    This statement holds as the determinant of any principal submatrix of $H^R$ or order two is equal to $-2$.
			
 To establish  whether for a given order $n$ there exists a robust Hadamard matrix
it is therefore sufficient to find a skew Hadamard matrix or a 
symmetric conference matrix of size $n$.
Concerning odd dimensions we show in Appendix C and D that robust Hadamard matrices do not exist
for any odd $n$.

\section{Robust Hadamard matrices and unistochasticity}\label{section:robust_and_unistochasticity}

Robust Hadamard matrices, introduced in the previous 
Section, will be used to 
investigate the problem, whether a given  bistochastic matrix is unistochastic. 
Before demonstrating such an application let us recall the necessary notions.

	\subsection{Definitions}

	\begin{definition}\label{def:Birkhoff-polytope-of}
	The set $\mathcal{B}_{n}$ of bistochastic matrices of size $n$ consists 
of matrices with non-negative entries which satisfy two sum conditions:

 $$\mathcal{B}_{n}=\{B\in M_{n}(\mathbb{R}), \
\sum_{i=1}^{n}B_{ij}=\sum_{i=1}^{n}B_{ji}=1,\ j=1,\dots,n   \}$$.
		\end{definition}
		Due to the Birkhoff theorem \cite{Bi46}, 
  the set $\mathcal{B}_{n}$ is equal to the
 convex hull of all permutation matrices of order $n$. 
		
		\begin{definition}\label{def:unistochastic}
	A bistochastic matrix $B\in\mathcal{B}_{n}$ is called unistochastic if and only 
   if there exists a unitary matrix $U\in U(n)$, $UU^{*} = \mathbb{I}$, such that
	\[
		B_{ij}=|U_{ij}|^{2},\ i,j=1,2,\dots,n
	\]
\end{definition}		
Using the notion of the Hadamard product $\circ$ of two matrices 
one can say that $B$ is unistochastic, 
if there exist a unitary $U$ such that $B=U \circ {\bar U}$.

If matrix $U$ is orthogonal the matrix $B$ is called {\sl orthostochastic}.
		The sets of all unistochastic and orthostochastic matrices
 of order $n$ will be denoted by $\mathcal{U}_{n}$ and $\mathcal{O}_{n}$, respectively.
 Above definitions imply:

		\[
		\mathcal{O}_{n}\subset\mathcal{U}_{n}\subset\mathcal{B}_{n}.
		\]
        
\subsection{In search for unistochasticity}
			It is easy to see that for $n=2$ all three sets coincide, 
$\mathcal{O}_{2} = \mathcal{U}_{2} = \mathcal{B}_{2}$. 
While conditions for unistochasticity are known \cite{AYP79,Di06,DZ09} for $n=3$, 
the case $n \geq 4$ remains open -- see  \cite{BEKTZ05}.
Given an arbitrary bistochastic matrix $B$ of order $n$, 
it is easy to verify whether certain necessary conditions for unistochasticity are satisfied,
but in general, no universal sufficient conditions are known
and one has to rely on numerical techniques \cite{BEKTZ05,GT17}.

Since we are looking for a unitary $U$ such that $B=U \circ {\bar U}$,
the absolute value of each entry is fixed, $|U_{ij}|=\sqrt{B_{ij}}$,
and one can investigate constraints implied by the unitarity.
The scalar product of any two different columns of $U$
is given by a sum of $n$ complex numbers of the form $U_{ij}{\bar U}_{ij'}$.
This sum can vanish if the largest modulus is not larger than
the half of the sum of all $n$ moduli. For $n=3$ this condition
is equivalent to the triangle inequality.  
The above observation, allows one 
to obtain a set of necessary conditions, a bistochastic $B$ must satisfy to
be unistochastic \cite{PZK01},
related to orthogonality of columns,
\begin{equation}
\max_{m=1,\ldots,n} \sqrt{B_{mk}B_{ml}}~\le ~ \frac{1}{2}
 \sum_{j=1}^n \sqrt{B_{jk}B_{jl}} \ ,{\rm ~~for ~~ all ~~} 1\leq k<l\leq n,
\label{chain_general1}
\end{equation}
and rows of $U$,
\begin{equation}
\max_{m=1,\ldots,n} \sqrt{B_{km}B_{lm}}~\le ~ \frac{1}{2}
 \sum_{j=1}^n \sqrt{B_{kj}B_{lj}} \ , {\rm ~~for ~~ all ~~} 1\leq k<l\leq n \ .
\label{chain_general2}
\end{equation}

We shall refer to the above relations as polygon inequalities or "chain" conditions: 
the longest link of a closed chain cannot be longer than the sum of all other links.
For matrices of size $n=3$ these necessary conditions are known
to be sufficient for unistochasticity \cite{AYP79,Na96},
while $B$ is orthostochastic if the bounds are saturated.
Furthermore, in this dimension all chain conditions are equivalent --
if inequality (\ref{chain_general1})  is satisfied for any pair of two columns 
it will also be satisfied  for the remaining 
pairs of columns and vectors \cite{JS88,Di06}.
The case $n=4$ occurs to be more complicated
as the chain inequalities are not sufficient \cite{AMM91,BEKTZ05},
and examples of bistochastic matrices are known
for which these condition are satisfied for some pairs of columns or rows 
and not satisfied by the other pairs \cite{Pa02,ZKSS03}. 
An algorithm allowing one to establish numerically
whether a given bistochastic matrix $B$ of order $n=4$ is unistochastic
is described in the Appendix A.

Not being able to solve the unistochasticity problem in its full extent, we shall consider 
particular subsets of the Birkhoff polytope $\mathcal{B}_{n}$ of an arbitrary dimension $n$. 
Figure 1 presents a sketch of the set of bistochastic matrices visualizing the problems considered. 
	
		\begin{figure}[H]
				\centering
	\includegraphics[scale=0.37]{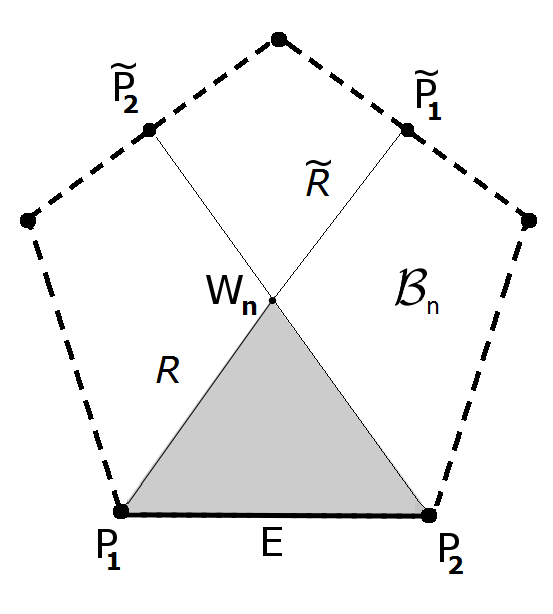}
\caption{Sketch of the Birkhoff polytope $\mathcal{B}_{n}$ - a set of ${(n-1)^2}$ dimensions: 
the flat matrix $W_n$ in the center, permutation matrices $P_{i}$ at the corners.
The corresponding rays $\mathcal{R}$  
and counter-rays  $\widetilde{\mathcal{R}}$  meet at $W_n$. The counter-ray 
 $\widetilde{\mathcal{R}}$ ends at the counter-permutation matrix $\tilde{P}_1$
 Bold lines denote complementary edges, while dashed lines represent non-complementary edges.
Region in gray -- the triangle $\Delta P_1, P_2, W_n$ -- represents sets proved to be unistochastic.}
\label{Fig1}
		\end{figure}
		
\begin{definition}
	Bistochastic matrix $W_n$, in which every element is equal to $1/n$, is called the flat matrix.
\end{definition}
\begin{definition}
	A one-dimensional set of bistochastic matrices 
obtained by a convex combination of the flat matrix $W_n$ and any permutation matrix $P$ is 
called a {\sl bistochastic ray},
\begin{equation}
\mathcal{R}_{n}=\{R_{\alpha}\in\mathcal{B}_{n}:\,R_{\alpha}= \alpha P +(1-\alpha) W_n,\ P\in\mathcal{P}(n), \ \alpha \in[0,1] \}.
\label{def_ray}
\end{equation}
\end{definition}
\begin{definition}
	A set $\widetilde{\mathcal{R}}_{n}$ of bistochastic matrices belonging to the line 
joining the flat matrix \emph{$W_n$ } and a permutation matrix $P$, outside the segment $PW_n$ is called a counter-ray. 
 These matrices can be expressed as pseudo-mixtures 
(\ref{def_ray}) with a negative weight $\alpha\in [-\frac{1}{n-1},0)$.
\end{definition}
For $P=\mathbb{I}$, the family of matrices belonging to the ray is as follows
\begin{equation}
	\mathcal{R}_{\mathbb{I}}=
    \Bigg \{
    R_{\alpha} = 
    \begin{pmatrix}a & b & \ldots & b & b\\
	b & a & \ldots & b & b\\
	\vdots & \vdots & \ddots & \vdots & \vdots\\
	b & b & \ldots & a & b\\
	b & b & \ldots & b & a
	\end{pmatrix} ,
    a=\frac{(n-1)\alpha+1}{n}; \text{ } b=\frac{1-a}{n-1};\text{   } \alpha \in[0,1]
    \Bigg\}.
    \label{eq:bista}
\end{equation}

\subsection{Robust Hadamard matrices imply unistochasticity}
	
	\begin{lemma}
If there exists a robust complex Hadamard matrix  $H^{R} $ of order $n$
then all rays and counter-rays of the Birkhoff polytope $\mathcal{B}_{n}$ of order $n$ are unistochastic.
	\end{lemma}
	\begin{proof}
	First let us show, that for any robust Hadamard matrix ${H}^{R}$ the following property holds:
	\begin{equation}
		{H}^{R}{D}^{*}+D{{H}^{R}}^{*}=2\mathbb{I},
	\label{eq:sskew}
	\end{equation}
	where $D$ is the diagonal matrix containing diagonal entries of $H^{R}$. Diagonal elements of the left-hand side of eq.~(\ref{eq:sskew}) are equal to $2H_{ii}H_{ii}^{*}=2$, while off-diagonal entries of this sum, $i \neq j$, 
	     read $H_{ij}{H_{jj}}^{*}+{H_{ji}}^{*}{H}_{ii}$.
    In order to evaluate these terms we shall use the fact that the matrix ${H}^{R}{D}^{*} \sim {H}^{R}$
    is robust.  
     Let $M_2$ be the principal submatrix of ${H}^{R}{D}^{*}$ of order two spanned by the rows $i$ and $j$. 
     Since ${H}^{R}{D}^{*}$ is robust then  
     $M_2 \in \mathcal{H}_{\mathbb{C}}(2)$,  so that $M_2M_2^*=2\mathbb{I}_{2}$. 
     Writing down  the entries of this matrix explicitly one obtains
     	\begin{equation*}
		 M_{2}{M_{2}}^{*}=
         \begin{pmatrix}2 & H_{ij}{H_{jj}}^{*}+{H_{ji}}^{*}{H}_{ii}\\
                  H_{ij}^*{H_{jj}}+{H_{ji}}{H}_{ii}^* & 2
		\end{pmatrix} .
	\end{equation*}
	 Since this matrix is proportional to identity its off-diagonal entries vanish, 
	  so that    $H_{ij}{H_{jj}}^{*}+{H_{ji}}^{*}{H}_{ii}=0,$ for $i \neq j$. 
	  Hence equation (\ref{eq:sskew}) is shown to be true.
	
	For any bistochastic matrix $R$ of order $n$ belonging to the ray $\mathcal{R}_{\mathbb{I}}$ 
	or counter ray $\widetilde{\mathcal{R}}$  let us now construct a matrix $U$,
	\begin{equation*}
		U=\sqrt{a}D+\sqrt{b}(H^{R}-D),
	\end{equation*}
	 where real parameters $a$ and $b$ defined by eq.~(\ref{eq:bista}) 
	  satisfy the condition $(n-1)b+a=1$.
	Making use of these relations,  definition ~(\ref{definition:hadamard}) and eq.~(\ref{eq:sskew}) 
	  we get  
	\begin{gather*}
		UU^{*}= (\sqrt{b}H^{R}-(\sqrt{b}-\sqrt{a})D)(\sqrt{b}{H^{R}}^{*}-(\sqrt{b}-\sqrt{a}){D}^{*})= \\
		bH^{R}{H^{R}}^{*}-\sqrt{b}(\sqrt{b}-\sqrt{a})(H^{R}{D}^{*}+D{H^{R}}^{*})+(\sqrt{b}-\sqrt{a})^{2}D{D}^{*}=((n-1)b+a)\mathbb{I}=\mathbb{I}.
	\end{gather*}
	   This shows that the matrix $U$ is unitary.
	   Since the matrix $R$ satisfies the relation $R_{ij}=|U_{ij}|^{2}$, it is unistochastic.
	     Hence any matrix  $R$ at any ray $\mathcal{R}$ of the Birkhoff polytope $\mathcal{B}_{n}$ 
	     or any counter-ray   $\widetilde{\mathcal{R}}$   is unistochastic 
	      for any dimension $n$, for which a robust Hadamard matrix exists.
\end{proof}
    
    In particular, if the robust Hadamard matrix is real, so that $U$ becomes orthogonal, then the matrix $R$ is orthostochastic. Since every skew Hadamard matrix is robust (lemma \ref{lemma:all_robust}), we arrive at the following statements:
    
	\begin{proposition}\label{proposition:sym}
For every order $n$, for which there exists a symmetric conference matrix, 
every matrix belonging to any ray $\mathcal{R}$ or any counter-ray $\widetilde{\mathcal{R}}$ of the Birkhoff polytope $\mathcal{B}_{n}$  is unistochastic.
	\end{proposition}    
\begin{proposition}
\label{proposition:skeww}
For any order $n$, for which there exists a skew Hadamard matrix $H^{S}$,
every matrix belonging to any ray $\mathcal{R}$ or any counter-ray $\widetilde{\mathcal{R}}$ 
of the Birkhoff polytope $\mathcal{B}_{n}$ is orthostochastic.
		\end{proposition}

Existence of skew Hadamard matrices for orders $n=4k$ is proved for $k<69$, proper construction was done by Paley. There are infinitely many cases of skew Hadamard matrices of higher orders \cite{skew}.

It is known that for dimensions $n=6,10,14,18$ there exists a symmetric conference matrix \cite{sym}. However, 
for order $n=22$ there are no such matrices, since $21$ is not the sum of two squares. 

	Those facts imply the main result of this work:
\medskip
	\begin{theorem}
		For any even order $n\leq 20$ all rays and counter-rays of the Birkhoff polytope $\mathcal{B}_{n}$ are 
		\begin{enumerate}[a)]
			\item orthostochastic (for $n=2,4,8,12,16,20$) or 
			\item unistochastic (for $n=6,10,14,18$).
		\end{enumerate}
	\end{theorem}
\medskip
	\begin{proof}
		It is a simple conclusion from proposition \ref{proposition:sym} and proposition {\ref{proposition:skeww}} which allow one for an explicit construction of the 
corresponding orthogonal and unitary matrices.
	\end{proof}
    Note that the above statement holds also for infinitely many dimensions, for which symmetric conference matrices are known (first found by Paley -- see e.g. \cite{Horadam07}).

\section{Unistochasticity of certain triangles embedded inside Birkhoff polytope}\label{section:triangles}
			A convex combination of any two permutation matrices forms an edge (or a diagonal) of the Birkhoff polytope $\mathcal{B}_n$,
			\begin{equation*}	
				E = \alpha P + (1-\alpha) Q
			\end{equation*}
		Au-Yeng and Cheng introduced the notion of complementary permutations {\cite{convex}}.

\begin{definition}
	Let $P = (P_{ij})$ and $Q = (Q_{ij})$ be two permutation matrices. 
   Then the matrices  $P$ and $Q$ are called {\sl complementary} if equality 
$P_{ij} = P_{hk} = Q_{ik} = 1$  implies $Q_{hj} = 1$  for all $i, j, h, k \in \{1,...,n\}$;
(consequently, if $Q_{ij} = Q_{hk} = P_{ik} = 1$ then $P_{hj} = 1$).
\label{def_comp}
	\end{definition}
		\begin{proposition}
If $P$ and $Q$ are two complementary permutation matrices
then the entire edge $PQ$ of the Birkhoff polytope $\mathcal{B}_n$ connecting $P$ and $Q$ is orthostochastic. 
If $P$ and $Q$ are not complementary, then the edge is not unistochastic, 
   besides the orthostochastic corners $P$ and $Q$.
		\end{proposition}
Below we generalize above proposition, proved by Au-Yeng and Cheng \cite{convex} to establish 
unistochasticity of some sets of larger dimension.
Let us first distinguish the following notion:

\begin{definition}
Two permutation matrices  $P$ and $Q$ will be called {\sl strongly complementary}
 if they are complementary and if the condition $P_{ij} = 1$ implies $Q_{ij} = 0$.
\label{def_strong}
	\end{definition}
In other words their non-zero elements are put in different places, 
so their Hadamard product vanishes, $P\circ Q=0$.
Due to symmetry of $\mathcal{B}_{n}$ 
one can take the identity matrix for the permutation matrix $P$ 
 without loosing generality.
A matrix  $Q$ is strongly complementary to $P=\mathbb{I}$ if and only if 
$Q$ is an involution, $Q^2=\mathbb{I}$,   
and every diagonal element of $Q$ is $0$. 
Due to complementarity, the dimension of $Q$ is even.
Making use of the flat bistochastic matrix $W_n$ we can now formulate statements
concerning triangles $\Delta(A,B,C)$ belonging to $\mathcal{B}_{n}$ 
and spanned by vertices $A$, $B$ and $C$. 
	
	\begin{lemma}
Consider dimension $n$ for which a robust Hadamard matrix $H^R$ exists. 
   If  $P$ and $Q$ are strongly complementary permutation matrices,
   then the triangle  $\Delta(P, Q, W_n)$ is unistochastic.
If  $H^R$ is real then the triangle contains orthostochastic matrices.
	\end{lemma}
	\begin{proof}
		As before, we can restrict our attention to the case $P = \mathbb{I}$. 
       Every matrix on the line connecting $\mathbb{I}$ and $Q$ can be written as:
	\[
		\begin{pmatrix}
			b & a & 0 & 0 & \dots\\
			a & b & 0 & 0 & \dots \\
			0 & 0 & b	& a & \dots\\
			0 & 0 & a & b & \dots\\
			\vdots & \vdots & \vdots & \vdots & \ddots & \\
		\end{pmatrix},
	\]
	where $a+b=1$. Thus, any bistochastic
matrix belonging to the triangle  $\Delta(P, Q, W_n)$ reads
	\[
		\begin{pmatrix}
		b & a & c & c & \dots\\
		a & b & c & c & \dots \\
		c & c & b & a & \dots\\
		c & c & a & b & \dots\\
		\vdots & \vdots & \vdots & \vdots & \ddots & \\
		\end{pmatrix},
	\]
		where $a+b+(n-2)c=1$. 
Taking the entry-wise square root of this matrix and multiplying it element-wise by a robust Hadamard matrix 
we obtain the corresponding unitary matrix, sufficient to show the desired property.
	\end{proof}

This statement is visualized by a gray triangle 
       of unistochastic matrices shown in  Fig. \ref{Fig1}.
The above result can be generalized for a larger set of permutation  
matrices $\{ P_1, P_2, ..., P_k \}$ of dimension $n$,  
in which each pair of matrices is strongly complementary.
Then an analogous reasoning shows that the bistochastic matrices belonging to 
$2$-faces of the polytope 
defined by the convex hull of $k$ these permutation matrices and the flat matrix $W_n$ are unistochastic.

Due to the definition of strong complementarity each pair of matrices has non-zero entries in different places,
so that  $k \leq n$. 
We are not able to determine how the maximal number $k$ of such matrices 
depends on the dimension $n$, nor whether the bistochastic matrices  
belonging to the interior of this polytope are unistochastic.
			
\section{Equi-entangled bases}\label{section:equi}
Any quantum system composed from two subsystems, with $n$ levels each,
can be described in a Hilbert space with a tensor product structure.
It is often natural to use the standard, product basis, usually denoted
by $|k,m\rangle=|k\rangle \otimes |m\rangle$ with $k,m=1,\dots, n$.
However, for certain problems it is advantageous to use bases consisting of 
maximally entangled states. In the simplest case of $n=2$ one uses 
the Bell basis consisting of four orthogonal states of size four,
$|\Psi^\pm\rangle =(|0,0\rangle \pm |1,1\rangle)/\sqrt{2}$ and
$|\Phi^\pm\rangle =(|0,1\rangle \pm |1,0\rangle)/\sqrt{2}$.
These {\sl Bell states} are also called maximally entangled, 
as their partial traces are maximally mixed. In this Section we follow the notation of \cite{We01}, often used in quantum theory.

For any higher dimension $n$ such entangled bases 
were constructed by Werner \cite{We01}.
A slightly more general variant of this construction discussed in {\cite{notes}}
allows one to write a set of $n^2$ orthogonal states related 
to a given unitary matrix $U$ of order $n$.
Making use of unistochastic matrices belonging at a ray
of the Birkhoff polytope and  determined by robust Hadamard matrices,
we shall construct a family of bases interpolating 
between separable and maximally entangled basis.

Consider a unitary matrix $U$ of order $n$, 
associated with the unistochastic matrix $B=U \circ {\bar U}$,
which belongs to the ray $R$ of the Birkhoff polytope ${\mathcal B}_n$
-- see Fig. \ref{Fig1}. 
It exists for all dimensions, for which a robust Hadamard matrix
exists (i.e. for all even dimensions $n \le 20$)
and allows us to construct an equi-entangled orthogonal basis in the Hilbert space
describing a composed system of size $n \times n$.

\begin{proposition}
Let  $|m,m'\rangle$ with $m,m'=1,\dots n$ form the standard
computation basis in the bipartite Hilbert space of size $n \times n$.
Let $U$ be a unitary matrix of order $n$, 
associated with the unistochastic matrix $B=U \circ {\bar U}$, which belongs to
a ray $R$ of the Birkhoff polytope ${\mathcal B}_n$.
Then the set of $n^2$ vectors in this space defined by 
\begin{equation}
|\psi_{m,k}\rangle  = \sum_{j=1}^n U_{mj}\; |j \rangle \otimes |j+k_{| {\rm mod}\ n} \rangle , \ 
\ k,m=1,\dots, n
\label{vectors}
\end{equation}

a) forms an orthonormal basis;

b) the basis is equi-entangled, as all its elements have the same degree of entanglement.
		\end{proposition}

		\begin{proof}
a) To show this property it is sufficient to check the scalar product
\begin{multline}
\langle \psi_{m,k}|\psi_{m',k'}\rangle 
=\sum_{j,j'} U^{*}_{mj}  \langle j,j+k_{| {\rm mod}\ n}|j',j'+k'_{| {\rm mod}\ n}\rangle U_{m'j'} = \\
    = \sum_{j,j'}U^{*}_{jm}U_{m'j'}\delta_{k,k'}\delta_{j,j'}
    = \delta_{k,k'}\sum_{j} U_{m'j}U^{*}_{jm} = \delta_{k,k'}\delta_{m,m'}.
\end{multline}
Due to unitarity of $U$ 
the orthogonality  holds for any $k,k',m,m'=1,\dots,n$,
 which implies that the vectors  (\ref{vectors}) form an orthonormal basis.

b) To analyze the degree of entanglement we specify our considerations
to a family $U(a)$ of unitary matrices
 given in ~(\ref{eq:bista})   
  and  parameterized by a single parameter $a$. 
A state $|\psi_{m,k}\rangle$ defined in eq.~(\ref{vectors}) 
can be rewritten in a slightly different way,
\begin{equation}\label{eq:basis}
	|\psi_{m,k}\rangle  = \sum_j |U_{mj}|\; |j\rangle_1 \otimes |j\rangle_2,
\end{equation}
where $|j\rangle_1$ is an element of the computational basis in the first system,
while $|j\rangle_2 =  |j+k_{| {\rm mod}\ n}\rangle$ 
  is obtained by relabeling the second basis.
As any quantum state is defined up to a complex phase we are allowed to replace the complex prefactor $U_{mj}$ by its modulus $|U_{mj}|$.

Note that the above form can be interpreted as the Schmidt decomposition of the bipartite state,
$|\psi_{AB}\rangle =\sum_{j=1}^n \sqrt{\lambda_j} |j\rangle_A \otimes |j\rangle_B$.
Observe that the components of the Schmidt vector, $\lambda_j=|U_{mj}|^2$,
form a row of the bistochastic matrix $B_a$ belonging to a 
ray of the Birkhoff polytope ${\mathcal B}_n$.
For each basis state $|\psi_{m,k}\rangle$ its ordered Schmidt vector is the same,
$\lambda=(a,b,b,b,\dots,b)$ with $b=(1-a)/(n-1)$, so that
  all measures of entanglement of the states $|\psi_{m,k}\rangle$ are equal.
Therefore the orthonormal basis parameterized by a number $a\in [1/n,1]$ 
forms a family of equi-entangled bases, which interpolate
between a maximally entangled basis ($a=1/n$) and a separable basis ($a=1$).
\end{proof}

To characterize the degree of entanglement one can apply
the entanglement entropy of a basis state,
equal to 
the Shannon entropy of the corresponding Schmidt vector
and to  the von Neumann entropy of the reduced density matrix.
As the Schmidt vector of each state has the structure
$\lambda=(a,b,b,b,\dots,b)$ its entropy reads
\begin{equation}\label{eq:entropy}
S(\lambda) = - a \log a 
             -(1-a) \log \frac{1-a}{n-1} ,
\end{equation}
and yields the entanglement entropy of the basis states (\ref{eq:basis}).
It is equal to zero for $a=1$ and a separable basis
and it achieves the maximum equal to $\log n$ for $a=1/n$, 
which corresponds to the maximally entangled basis, analyzed in \cite{We01}.

\section{Set ${\mathcal U}_4$ of unistochastic matrices of order $n=4$}\label{section:set_dimension4}
For even dimensions $n\le 20$ we have shown that
all rays of the Birkhoff polytope are unistochastic. This statement
holds also in the case $n=4$ -- the smallest dimension for which the
unistochasticity problem is still open \cite{AMM91}.
In this case the chain conditions (\ref{chain_general1}) and (\ref{chain_general2}) are necessary, 
but in contrast to the case $n=3$ 
they are not sufficient to assure unistochasticity \cite{BEKTZ05}.

Although analytic form of such conditions remains still not known,
we shall apply a numerical procedure proposed by Haagerup -- see Appendix A --
to study the properties of the set 
${\mathcal U}_4$ of unistochastic matrices of order $n=4$.
Generating random bistochastic matrices according to the flat measure in 
the Birkhoff polytope ${\mathcal B}_4$ according to the algorithm 
described in \cite{CSBZ09}
we found that the
relative volume of the set ${\mathcal C}_4$ of bistochastic matrices 
satisfying all chain conditions (\ref{chain_general1}) and (\ref{chain_general2}) 
is $V_C\approx 0.71$,
 while the volume of its subset ${\mathcal U}_4$ containing unistochastic matrices is 
 $V_U\approx 0.61$ in comparison to the total volume of ${\mathcal B}_4$. 

However, since not much is known about geometric properties of the subsets ${\mathcal U}_4 \subset \mathcal{C}_4 \subset \mathcal{B}_4$ of the Birkhoff polytope, we shall study various cross-sections which include the flat matrix $W_4$ located at its center.

Any such cross-section is determined by three matrices which do not belong to a single line.
Since one of these matrices is selected to be $W_4$, we have to specify only two other matrices.
Cross-sections shown in Fig.  \ref{FigMerged} are specified by two permutation matrices.
Due to the symmetry of the polytope ${\mathcal B}_4$  we may take the first
one as identity, $\mathbb{I}_4$, without loosing the generality.

There exist different kinds of edges in ${\mathcal B}_4$ labeled by their length in sense of the Hilbert-Schmidt distance, $D(X,Y) = \sqrt{{\rm Tr}[(X-Y)^2]}$.
One distinguishes \cite{BEKTZ05} following four classes of edges:
\begin{enumerate}[a)]
			\item unistochastic short edges,  $\mathbb{I}_4 P_A$ of length $2$, where  ${\rm Tr}P_A=2$,
			\item not unistochastic middle edges,  $\mathbb{I}_4 P_B$ of length $\sqrt{6}$, where  ${\rm Tr}P_B=1$,
            \item not unistochastic long edges,  $\mathbb{I}_4 P_C$ 
                           of length $\sqrt{8}$,  where  ${\rm Tr}P_C=0$ and $P_C^2 \ne \mathbb{I}_4$,
            \item unistochastic long edges,  $\mathbb{I}_4 P_D$ 
                           of length $\sqrt{8}$,  where  ${\rm Tr}P_D=0$ and $P_D^2 = \mathbb{I}_4$.

\end{enumerate}

Cross-sections determined by unistochastic edges of length $\sqrt{8}$ are simple as all bistochastic matrices are unistochastic, which follows from  Section \ref{section:triangles}. Below we present cross-sections determined by edges from remaining three classes, created by the following permutation matrices (the numbers in the subscripts denote the cycles):

\[
		P_A = P_{12} = \begin{pmatrix}
			0 & 1 & 0 & 0 \\
			1 & 0 & 0 & 0  \\
			0 & 0 & 1	& 0 \\
			0 & 0 & 0 & 1 \\
		\end{pmatrix},\text{ }
        P_B = P_{123}= \begin{pmatrix}
			0 & 0 & 1 & 0 \\
			1 & 0 & 0 & 0 \\
			0 & 1 & 0	& 0 \\
			0 & 0 & 0 & 1 \\
		\end{pmatrix}, \text{ }
        P_C = P_{1234} = \begin{pmatrix}
			0 & 0 & 0 & 1 \\
			1 & 0 & 0 & 0  \\
			0 & 1 & 0 & 0 \\
			0 & 0 & 1 & 0  \\
		\end{pmatrix}.
	\]

 \begin{figure}[H]
				\centering
	\includegraphics[scale=0.14]{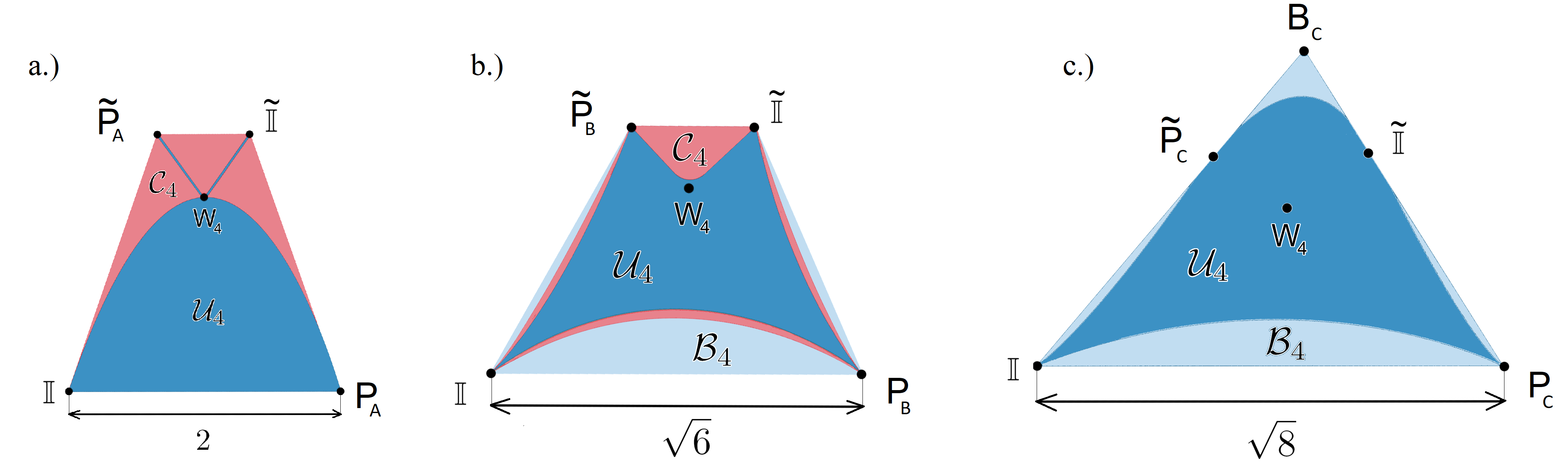}
\captionsetup{singlelinecheck=off}
	\caption[]{Cross-sections of the Birkhoff polytope $\mathcal{B}_{4}$: 
the flat matrix $W_4$ in the center, determined by selected edges:
     \begin{enumerate}[a)]
			\item unistochastic  short edge $\mathbb{I} \leftrightarrow P_A$ of length $2$
			\item not unistochastic middle edge $\mathbb{I} \leftrightarrow P_B$ of length $\sqrt{6}$
            \item not unistochastic long edge $\mathbb{I} \leftrightarrow P_C$ of length $\sqrt{8}$
\end{enumerate}
 Darkest gray (in color: dark blue) represents unistochastic set $\mathcal{U}_4$, medium gray (red) represents the set $\mathcal{C}_4$ of matrices that satisfy the chain conditions, and the lightest gray (light blue) denotes bistochastic matrices, that do not satisfy chain conditions. At the section showed in 
 panel c) the  sets $\mathcal{U}_4$ and $\mathcal{C}_4$ do coincide, 
 while the bistochastic matrix at the edge reads
 $B_c= 2 W_4 - \frac{1}{2} P_c-\frac{1}{2}\mathbb{I}_4$.
Note the dark sets  $\mathcal{U}_4$ in all panels containing both counter-rays 
are not convex but are star-shaped.}
\label{FigMerged}
		\end{figure}

To produce each plot we generated a lattice of around $10^4$ bistochastic
matrices belonging to a given 2D cross--section and verified
if conditions (\ref{chain_general1}) are satisfied and
whether numerical procedure described in Appendix A returns 
the corresponding unitary matrix.
        The Birkhoff polytope $\mathcal{B}_{4}$ possesses an interesting property -- in every neighborhood of the flat unistochastic matrix $W_4$ localized at its center there are non-unistochastic matrices. Hence there is a direction, in which deviation by an arbitrary small $\epsilon$ leads to a matrix which is not unistochastic \cite{BEKTZ05}.
        Selecting one of those directions according to \cite{BEKTZ05} we find a cross-section shown in Fig. 3, such that non-unistochastic matrices are located arbitrary close to $W_4$:
        
        \begin{figure}[H]
				\centering
	\includegraphics[scale=0.2]{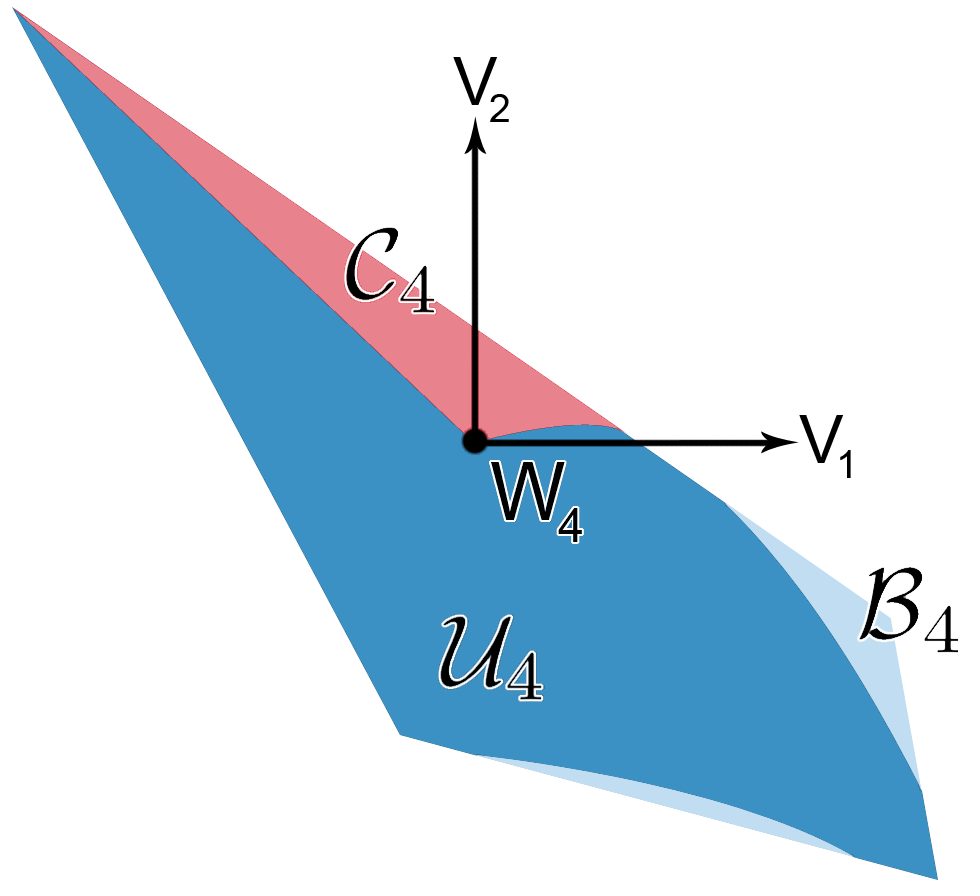}
\caption{Cross-section of the Birkhoff polytope $\mathcal{B}_{4}$, determined by the  directions $V_1$ and $V_2$ described in \cite{BEKTZ05}. The flat matrix $W_4$ in the center of the $\mathcal{B}_4$ is also localized at the boundary of the set $\mathcal{U}_4$.
Colors have the same meaning as in Fig. 2.}
\label{Fig5}
		\end{figure}

Our findings do not contradict the conjecture, that the set ${\mathcal U}_4$ of unistochastic matrices is star-shaped with respect to $W_4$.
There exist 2D cross-sections, in which the sets ${\mathcal U}_4$ and ${\mathcal C}_4$ coincide, but it is easy to find a cross-section which reveals that the inclusion relation $\mathcal{U}_4 \subset \mathcal{C}_4$ is proper. Compare the discussion of relative volumes of these sets presented in Appendix B.

\section{Concluding remarks}

A notion of Hadamard matrices robust with respect to a projection onto a subspace formed by any two of the basis vectors is introduced.
In the case of a double even dimension, $n=4k$, existence of such matrices follows from 
existence of skew real Hadamard matrices. Furthermore, 
existence of symmetric conference matrices of dimension $n=4k+2$ with $k\ge 1$
yields a construction of a robust complex Hadamard matrix of this size.
This implies that robust Hadamard matrices exist for all even dimensions $n \le 20$.

Existence of robust Hadamard matrices of order $n$ allows us to show that all the rays
of the Birkhoff polytope ${\mathcal B}_n$ of bistochastic matrices
are unistochastic. Hence for any point $B_a$ at the line joining
an edge (a permutation matrix) with the flat matrix $W_n$ at the 
center of the polytope, 
there exists a corresponding unitary matrix $U(a)$ of size $n$
such that $B_a=U(a) \circ {\bar U(a)}$. Furthermore, if two permutation 
matrices $P$ and $Q$ are strongly complementary \cite{convex} -- 
see  Definition \ref{def_strong}  --
the entire triangle $\Delta(P,Q,W_n)$ is unistochastic.

The family of unitary matrices $U(a)$ of order $n$ obtained with help of 
robust Hadamard matrices of size $n$,
allows us to construct a family of equi-entangled bases in the 
composed Hilbert space of $n \times n$ system.
This family interpolates between the separable basis and the maximally entangled basis.
Each interpolating basis consists of $n^2$ normalized vectors with the same
Schmidt vectors, which  determine the entropy of entanglement (\ref{eq:entropy}). 
In contrast to the earlier constructions given in \cite{KM06} and extended in \cite{GL10}
our construction is based on a straight line in the Birkhoff polytope ${\mathcal B}_n$,
so the Schmidt vectors contain only two different entries.

This work contributes to investigations \cite{BEKTZ05}
of the set ${\mathcal U}_4$ of unistochastic matrices of order $n=4$.
Making use of the Haagerup algorithm to verify numerically  
whether a given bistochastic matrix $B \in {\mathcal B}_4$ 
is unistochastic -- see Appendix A --
we provided a first estimation of the relative volume of the set 
${\mathcal U}_4$ of unistochastic matrices and the larger set $\mathcal{C}_4$ of
these bistochastic matrices, which satisfy all chain conditions (3.1) and (3.2).
Furthermore, we studied the geometry of cross-sections of the set ${\mathcal U}_4$ 
along planes determined by selected permutation matrices --
corners of the Birkhoff polytope ${\mathcal B}_4$.

\medskip

Let us conclude the paper with a short list of open questions.

\begin{enumerate}
\item Given a Hadamard matrix $H$ of order $n$ is it possible to find
    an equivalent robust matrix, $H^{R}\sim H$?
  Due to Lemma \ref{lemma:equiv} this question is 
analogous to the problem concerning existence of skew Hadamard matrices.

\item What are the properties of the set of robust Hadamard matrices
    in dimension $n$ for which there exists
       several not equivalent Hadamard matrices $(n=16,\,20,\,\dots)$?

\item Is the convex hull of all mutually strong complementary permutation matrices 
of a fixed order $n$ and the flat matrix $W_n$ unistochastic? 

\item Is it possible to obtain necessary and sufficient conditions for unistochasticity
    applicable for any bistochastic matrix of order $n \geq 5$?
    
\item For any large dimension $n$, a generic bistochastic matrix is conjectured to meet all chain conditions -- see Appendix B. Following question remains open: what is the dependence of fraction of unistochastic matrices $f_u$ on dimension $n$ for $n \geq 5$?

\end{enumerate}
\medskip

{\bf Acknowledgements.}
One of the authors (K.{\.Z}.) had a chance to discuss the 
unistochasticity problem with the late Uffe Haagerup
during the conference in B\k{e}{}dlewo in July 2014, where he learned  about the
way to treat the $n=4$ case presented in the Appendix A.
It is a pleasure to thank Ingemar Bengtsson and Irina Dimitru
for numerous discussions on equi-entangled bases and for sharing with us their unpublished notes.
We are thankful to Dardo Goyeneche and Wojciech Tadej for numerous interactions 
and also to Robert Craigen and William Orrick
for fruitful discussions during the workshop in Budapest in July 2017. 
Special thanks are due to  Mate Matolcsi and Ferenc Sz\"{o}ll\H{o}si,
for organizing such a successful conference
which made these interactions possible.
We are deeply obliged to the referee for a long list of constructive comments, 
valuable hints to the literature and for suggesting us results presented in Appendix D.
Financial support by Narodowe Centrum Nauki
under the grant number DEC-2015/18/A/ST2/00274  is gratefully acknowledged.

\bigskip
\noindent
{\bf Appendix A.} 
{\bf Haagerup procedure for studying unistochasticity of $n=4$ matrices.}

\medskip
In this appendix we present a method due to the late Uffe Haagerup of searching for a unitary matrix corresponding to any unistochastic matrix of order $n=4$. 
Given a bistochastic matrix $B \in  \mathcal{B}_{4}$
we wish to find $U$ such that $B_{ij} = |U_{ij}|^2$.

Any matrix of order four can be decomposed into four blocks of order two. 
Let us then represent in this way the unitary matrix $U$ we are looking for, 
\begin{equation} U = \left( \begin{array}{c|c} A & X \\ \hline Y & D \end{array} 
\right) = \left( \begin{array}{llll} \sqrt{B_{11}} & \sqrt{B_{12}} & \sqrt{B_{13}} 
& \sqrt{B_{14}} \\ \sqrt{B_{21}} & \sqrt{B_{22}}e^{i\phi} & \sqrt{B_{23}}e^{i\alpha_1} 
& \sqrt{B_{24}}e^{i\alpha_2} \\ \sqrt{B_{31}} & \sqrt{B_{32}}e^{i\beta_1} & \cdot & 
\cdot \\ \sqrt{B_{41}} & \sqrt{B_{42}}e^{i\beta_2} & \cdot & 
\cdot \end{array} \right) \ ,
\label{Uffe}
 \end{equation}

Note that all the moduli are determined by the bistochastic 
matrix $B$, so only the phases remain unknown.
Unitarity  of $U$ implies that 
\begin{equation} UU^* = \mathbb{I} \hspace{5mm} \Rightarrow \hspace{5mm} 
\left\{ \begin{array}{l} AA^* + XX^* = \mathbb{I} \\ \\ AY^* + XD^* 
= 0 \ . \end{array} \right.
\label{unit}
\end{equation}

Due to existence of the Hadamard matrix $H_4$
we know that the flat matrix $W_4$ is orthostochastic,
so we can assume that $B\ne W_4$.
Permuting rows and columns of $B$ one can 
rearrange the matrix in such a way that the following relation holds
\begin{equation*} 
	B_{11} + B_{12} + B_{21} + B_{22} < 1 \ . 
\end{equation*}
After this is done the norm of  the block $A$ is bounded,
\begin{equation*} 
  ||A||^2_{\rm HS} < 1 \hspace{3mm} \Rightarrow \hspace{3mm} \mbox{eigenvalue}(AA^*) < 1 \hspace{3mm} \Rightarrow \hspace{3mm} 
  \mbox{eigenvalue}(XX^* ) > 0 \ , 
\end{equation*}
so the neighboring block $X$ is  invertible.
Hence we can use the matrix $X^{-1}$ to transform
the second unitarity condition in (\ref{unit})
into an explicit expression for the lower diagonal block $D$,
\begin{equation} 
	D = -YA^* (X^{*})^{-1} \ .
    \label{DDD}
\end{equation}
\noindent The next step is to find the phases in the blocks $X$ and $Y$. 
To take into account orthogonality between 
the first two rows of the matrix $U$,
it is convenient to introduce four auxiliary variables,
\begin{equation*} l_1 = \sqrt{B_{11}B_{21}} \ , \hspace{4mm} 
l_2 = \sqrt{B_{12}B_{22}} \ , \hspace{4mm} l_3 = \sqrt{B_{13}B_{23}} \ , \hspace{4mm} 
l_4 = \sqrt{B_{14}B_{24}} \ . 
\end{equation*}

They allow us to rewrite this orthogonality condition,
\begin{equation}
	l_1 + l_2e^{i\phi} + l_3e^{i\alpha_1} + l_4e^{i\alpha_2} = 0 \,
    \label{chain}
\end{equation}
which can be treated as an equation for the unknown phases.
Orthogonality requires that a chain formed out of four links
of lengths $l_i$ has to be closed,
so the longest link is not longer than the sum of remaining links,

\begin{align} 
\label{links1}
l_i \leq \frac{1}{2} \sum^4_{j=1} l_j \quad \textrm{for} \quad i = 1,\dots,4.
\end{align}

If this condition is not satisfied, 
the matrix $B$ is clearly not unistochastic.
Observe that these conditions can be interpreted as particular cases
of the general conditions (\ref{chain_general1}) and (\ref{chain_general2}).

If  inequalities (\ref{links1}) are fulfilled 
there exist two solutions of eq.~(\ref{chain})
corresponding to a convex and a non-convex polygon,
\begin{equation*} \alpha_1 = \alpha_1(\phi) \ , \hspace{6mm} 
\alpha_2 = \alpha_2(\phi) \ 
 \end{equation*} 
which depend on the phase $\phi$ treated here as a free parameter.

Making use of the orthogonality condition 
between two first columns of $U$ we arrive at an
equation analogous to (\ref{chain}),
which for a given phase $\phi$
can be solved  for unknown phases $\beta_1$ and $\beta_2$. 
This determines the blocks $X$ and $Y$ and allows us to obtain
the remaining block $D$ of $U$.

The catch is that the explicit formula  (\ref{DDD})
produces a matrix $D$, which needs not to be 
compatible with the structure imposed by the initial 
bistochastic matrix $B$. To find the desired unitary matrix
$U$ it is sufficient to check that a single element
of $D$ has the correct norm. Hence we arrive at 
the following  criterion for unistochasticity  for a matrix  $B$ of order four:

\medskip
{\sl
A bistochastic matrix $B \in  \mathcal{B}_{4}$
is unistochastic if there exist a phase $\phi$ entering eq.~(\ref{Uffe}), 
(which determines phases $\alpha_i$ and $\beta_i$ and thus blocks $X$ and $Y$), 
such that the block $D$ obtained by eq.~(\ref{DDD})
satisfies the constraint $|D_{11}|^2=B_{33}$.}

\medskip
If this is the case the unitary matrix $U$ given by the form (\ref{Uffe})
satisfies the unistochasticity condition,
conveniently written by the Hadamard product, $B=U\circ {\bar U}$.
Note that the above procedure can be easily implemented numerically
for any given bistochastic matrix $B\ne W_4$ of order four.

\bigskip
\noindent
{\bf Appendix B.} 
{\bf Unistochasticity and chain link conditions in small dimensions -- numerical results}

\textbf{Random points in a cube.}
The simplest method to generate a random bistochastic matrix $B$ from the Birkhoff polytope $\mathcal{B}_{n}$ with respect to the uniform measure is to take random numbers, uniformly distributed in $[0,1]$ for each entry of the core -- the principal submatrix of order $n-1$, which determines the matrix $B$.
If the sum of any row or column of the
core is greater than $1$ or the sum of all elements of the core is
smaller than $n-2$, then the generated matrix is not
bistochastic and will not be considered. If all these conditions are met, we generate a bistochastic matrix by filling the missing row and column
with values which will add up all rows and columns to $1$. The entry $B_{nn}$ is equal to the
sum of all elements of the core reduced by $n-1$.

To optimize computational procedure of generating the cores, we use
row discrimination. Drawing the core row-by-row (rows are independent)
allows one to check the sum of the row at every step and stop the procedure if it exceeds unity.
This method occurs to work efficiently for dimensions
$n\leq6$, so for higher dimensions one needs to apply other techniques \cite{CSBZ09,CKKZ17}.

\textbf{Sinkhorn algorithm}
To generate a bistochastic matrix for $n \geq 7$ we also used a method described in \cite{CSBZ09}, analogous to the Sinkhorn algorithm \cite{sinkhorn1964}, which normalizes rows and columns of a given square matrix
in a following sequence:
\begin{enumerate}
\item input a random stochastic matrix of dimension $n$ with positive elements,
\item normalize every row by dividing it by the sum of its elements,
\item normalize every column of the matrix by dividing it by the sum of its elements,
\item go to point 2., unless the matrix is bistochastic up to a certain accuracy
with respect to the chosen norm.
\end{enumerate}
In practice we stopped the procedure if all the sums of rows and vectors are close to unity up to the sixth decimal place $(0.9999999\leq\varSigma rows,\varSigma columns\leq1.0000001)$.
This procedure occurs to be much faster than taking points at random from the core
for dimensions $n>6$.

\textbf{Dirichlet distribution.}
It is natural to study the measure induced in $\mathcal{B}_n$ by the Sinkhorn algorithm applied to random stochastic matrices with entries described by the Dirichlet distribution \cite{Dirichlet}, $P_s(x_1,\dots,x_n) \propto \prod^n_{i=1} x^{1-s}_i$, with the constraint $\sum_i x_i = 1$. For $s = 1$ this distribution coincides with the uniform distribution in the probability simplex. For a certain value of the parameter $s_*  = \frac{1}{2n^2}(n^2-2+\sqrt{n^4-4})$ the measure induced in the set of bistochastic matrices becomes uniform  \cite{CSBZ09} in the limit of large $n$. 

To generate a sequence of $n$ numbers distributed according to Dirichlet
distribution with a chosen parameter $s$ we used the following sequence:
\begin{enumerate}
\item generate $n$ independent random numbers from the gamma distribution of
rate $1$ and shape $s$. Explicitly, every element of $(x_{1},x_{2},\ldots,x_{n})$,
is drawn with probability density $\frac{1}{\Gamma(s)}x^{s-1}e^{-x}$.
\item normalize every element by dividing it by the sum of all elements, 
$x_{i} \to x'_i = \frac{1}{\underset{i}{\sum x_{i}}}x_{i}$.
\end{enumerate}

Applying the above procedure $n$ times we obtain random stochastic matrix of $n$ independent
rows each distributed according to Dirichlet distribution. Making use of the Sinkhorn algorithm we to obtain an ensemble of bistochastic matrices which is uniformly
distributed at the center of the Birkhoff polytope.

Numerical computations show that the first method (generating random points in a cube) is reliable for $n \leq 6$, while for $n > 6$
the Sinkhorn method, used with the Dirichlet parameter $s_*$, becomes more efficient.
A numerical estimation of the relative volume of the set of unistochastic matrices is shown in Fig.~\ref{Fig2}. Note that the relative volume of the set $\mathcal{C}_n$ of matrices satisfying the chain conditions approaches unity.

\begin{figure}[H]
				\centering
	\includegraphics[scale=0.8]{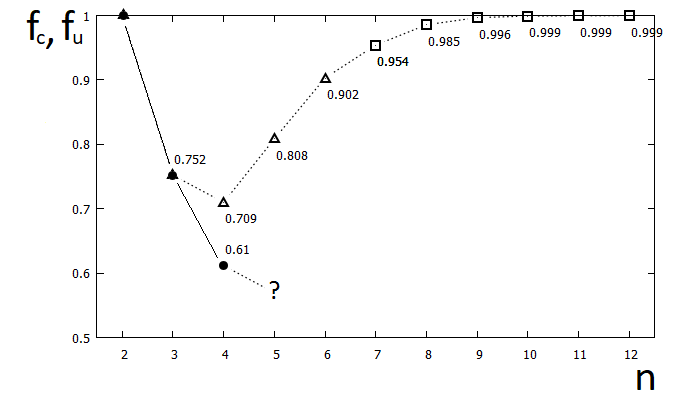}
\caption{Fraction $f_u$ of random bistochastic matrices, generated with respect to the flat measure, satisfying unistochasticity conditions (black circle -- $\bullet$) computed for dimensions $n < 5$.
Empty symbols (white triangle -- ${\bigtriangleup}$ and white square -- $\square$) denote the fraction $f_c$ of bistochastic matrices, which fulfill all chain conditions (3.1-3.2). Dotted line is plotted to guide the eye.}
\label{Fig2}
		\end{figure}
        
\bigskip
\noindent
{\bf Appendix C.} 
{\bf No robust Hadamard matrices of order $3$ and $5$}

To analyze whether a complex Hadamard matrix $H'$ of order $n$
 satisfies the equivalence relation (\ref{eq:equi_relations1}) with respect to a given Hadamard matrix $H$
one can use the set of invariants of Haagerup \cite{Ha96},
\begin{equation} 
\Lambda_{ij,kl}:= H_{ij}  {\bar H_{kj}} H_{kl} {\bar H_{il}},
\label{invH}
\end{equation}
where no summation over repeated indices is assumed
and $i,j,k,l=1,\dots, n$.
The complex numbers $\Lambda_{ij,kl}$, 
depending on the cumulative difference of phases,
are invariant with respect to multiplication of $H$ 
by diagonal unitary matrices.
Even though these $n^4$ complex numbers may be altered by permutations of rows and columns,
the entire set $\Lambda(H)=\{\Lambda_{ij,kl}\}_{i,j,k,l=1}^n$
remains invariant with respect to the equivalence relation (\ref{eq:equi_relations1}).

In the case of the Hadamard matrix $H_2$ the set of invariants
$\Lambda(H_2)$ contains the number $-1$ corresponding to the phase $\pi$.
Hence the set $\Lambda(H^R)$ for a robust Hadamard matrix $H^R$ of order $n$ has to 
contain at least $n(n-1)/2$ these entries. 

Consider now a Fourier matrix $F_n$ with $n>1$ and entries
\begin{equation}
  \mathcal(F_n)_{jk} = {e}^{ i (j-1)(k-1)2 \pi /n},
\end{equation}
with $j,k=1,\dots, n$,
which is a complex Hadamard matrix. 
If the dimension  $n$ is odd the set of Haagerup invariants
$\Lambda(F_n)$ does not include the number $-1$ as adding 
multiples of the basic phase $2 \pi/n$ one cannot obtain $\pi$.
This implies directly that for any odd number $n > 1$
the Fourier matrix $F_n$ and any equivalent complex Hadamard matrix
is not robust.
Since for $n=3$ and $n=5$ all complex Hadamard matrices are equivalent to
the Fourier matrix \cite{Ha96}, we arrive at the desired statement.

\medskip

{\bf Proposition.} 
{\sl There are no robust Hadamard matrices of dimension $n=3$ and $n=5$.}

\medskip

Alternatively, to show that there are no robust Hadamard matrices of order $3$ one can use the result of Cohn \cite{Cohn}, who establishes that "there is no
room for a sub-Hadamard matrix of size $m$ within a Hadamard matrix of size
$n$, where $m > n/2$."

\bigskip
\noindent
{\bf Appendix D.} 
{\bf On existence of robust Hadamard matrices}\\

In this Appendix we provide results concerning existence of robust Hadamard matrices which were suggested by the referee and go beyond the statements formulated in Appendix C. \\

Let $C$ be a complex conference matrix of order $n$, that is $C_{jj}=0$, $|C_{jk}|=1$ and $CC^{*}=(n-1)\mathbb{I}$.

\begin{lemma}
If $R$ is a robust Hadamard matrix of order $n \ge 2$, then $R$ is equivalent in the sense of eq.~(\ref{eq:equi_relations2}) to a matrix $H=C+i\mathbb{I}$, where $C$ is a self-adjoint complex conference matrix.
\label{lemma:robust_self-adjoint}
\end{lemma}
\begin{proof}
	The relation between $R$ and $H$ is explicitly:
	\begin{equation*}
		H=iRD^{*}
	\end{equation*}
where $D$ is the diagonal matrix containing diagonal entries of $R$ and $i$ is imaginary unit. All of the diagonal elements of $H$ are equal to $i$. Because D is a unitary matrix, then $R \sim H$ and thus 
$H$ is  a robust Hadamard matrix. This implies that any principal submatrix of H of order two,
	\[
			M_2 =\begin{pmatrix} i & h_{jk}\\
			h_{kj} & i \\
			\end{pmatrix}
		\]
has to satisfy $|{\rm det}(M_2)|=2$. Because $|h_{jk}|=1$ for any $j,k$, then $h_{jk}=h_{kj}^{*}$.
It implies that the matrix consisting of off-diagonal elements of $H$, namely $C=H-i\mathbb{I}$ is self-adjoint $(C=C^{*})$ and $|C_{jk}|=1$ for $j \ne k$. Self-adjointness of $C$, it implies that $H-H^{*}=2i\mathbb{I}$. Finally we can obtain:
	\begin{equation*}
		CC^{*}=(H-i\mathbb{I})(H^{*}+i\mathbb{I})=HH^{*}+i(H-H^{*})+\mathbb{I}=n\mathbb{I}-2\mathbb{I}+\mathbb{I}=(n-1)\mathbb{I}
	\end{equation*}
which completes the proof.
\end{proof}
Above result shows that robust complex Hadamard matrices are essentially a family of matrices equivalent to self-adjoint complex conference matrices with the diagonal elements filled with imaginary unit $i$.
\begin{lemma}
	If a robust Hadamard matrix $H$ of order $n$ exists, then $n$ is even.
\end{lemma}
\begin{proof}
	Applying Lemma \ref{lemma:robust_self-adjoint} we can restrict ourselves to the case $H=C+i\mathbb{I}$. Since $C=C^{*}$ and $C$ is unitary up to scalar, then its spectrum is real and the eigenvalues are equal to $\pm \sqrt{n-1}$. Note also that $C$ is traceless. Therefore, for $n \ge 2$, the number of positive and negative eigenvalues must be equal.  It follows that the number of eigenvalues, which is equal to $n$, must be even.
\end{proof}

\end{document}